\documentclass[reqno,12pt]{amsart} %
\usepackage{amsmath,amstext, amsthm, amssymb,amsfonts}

\numberwithin{equation}{section}

\def\RR{{\mathbb R}}
\def\CC{{\mathbb C}}

\def\eps{\varepsilon}
\newcommand{\la}{\lambda}

\newtheorem{theorem}{Theorem}[section]
\newtheorem{proposition}[theorem]{Proposition}

\newtheorem{corollary}[theorem]{Corollary}
\theoremstyle{definition}
\newtheorem{definition}{Definition}[section]
\newtheorem{remark}{Remark}[section]

\newcount\mins \newcount\hours \hours=\time \mins=\time
\divide\hours by60 \multiply\hours by 60 \advance\mins by-\hours
\divide\hours by60

\def\now{ \ifnum\hours>11 \ifnum\hours>12 \advance\hours by
-12 \fi \number\hours:\ifnum\mins<10 0\fi \number\mins\ pm,\ \else
\ifnum\hours=0 \hours=12 \fi \number\hours:\ifnum\mins<10 0\fi
\number\mins\ am,\ \fi}

\newcommand{\V}{\mathbb{V}}

\newcommand{\R}{\mathcal{R}}

\newcommand{\ToDoList}[1]{\shadowbox{\begin{minipage}{3.5in} #1 \end{minipage}}}
\newcommand{\Comment}[1]{\marginpar{\footnotesize #1}}
\renewcommand{\ToDoList}[1]{}
\renewcommand{\Comment}[1]{}

\newcommand{\CSK}{Cauchy-Stieltjes Kernel (CSK)\ \renewcommand{\CSK}{CSK\ }}

\title{One-sided Cauchy-Stieltjes Kernel Families}

\author{
W{\l}odzimierz  Bryc}\thanks{\noindent Research partially supported
by NSF grant  DMS-0904720
 and by Taft Research Seminar 2008/09}

\date{ Created May 10, 2009. Revised: June 19, 2009 and April 27, 2010.\\
{\em File:} {\tt \jobname .tex} of  \today, \now
}

\address{
Department of Mathematical Sciences,\\
University of Cincinnati,\\
PO Box 210025,\\
Cincinnati, OH 45221--0025, USA}
\email{Wlodzimierz.Bryc@UC.edu}

\author{Abdelhamid Hassairi}

\address{Faculty of Sciences\\
Sfax University \\
PO Box 1171, 3000, Sfax, Tunisia }
\email{abdelhamid.hassairi@fss.rnu.tn}


\begin{document}

\begin{abstract}
This paper continues the study of   a  kernel  family   which uses
the Cauchy-Stieltjes kernel $1/(1-\theta x)$ in place of the
celebrated exponential kernel $\exp({\theta x})$ of the exponential
families theory.  We extend the theory to cover generating measures
with support that is unbounded on one side. We illustrate the need
for such an extension by showing that cubic pseudo-variance
functions  correspond to free-infinitely divisible laws without the
first moment. We also
determine the domain of means, advancing  the understanding of
 Cauchy-Stieltjes kernel families
also for compactly supported generating measures.
\end{abstract}

\maketitle


\section{Introduction and definition}
According to Weso\l owski \cite{Wesolowski90}, the kernel family
generated by a kernel $k(x,\theta)$ with generating measure $\nu$ is the set of probability measures
$$\{k(x,\theta)/L(\theta)\nu(dx):\; \theta\in\Theta\},$$ where
$L(\theta)=\int k(x,\theta)\nu(dx)$ is the normalizing constant, and
$\nu$ is the generating measure.
The theory of exponential families is based on the kernel
$k(x,\theta)=e^{\theta x}$.  See, e.g., \cite{Jorg}, \cite{Let},  or
\cite[Section 2.3]{Diaconis-08}. Reference  \cite{Bryc-06-08}
initiated the study of the Cauchy-Stieltjes  kernel
\begin{equation*}
k(x,\theta)=\frac{1}{1-  \theta x},
\end{equation*}
  for compactly supported generating measures $\nu$. In the
 neighborhood of $\theta=0$, such families can be parameterized by
 the mean and under this parametrization, the family (and measure
 $\nu$) is uniquely determined by the variance
function $V(m)$ and a real number $m_0$, which is the mean of $\nu$. The  \CSK
  family has some properties analogous to classical exponential
 families, with convolution of measures replaced in some results by
 the free additive convolution.

 For measures with unbounded support, one can still
 define the \CSK family
 if the support of the generating measure is bounded above or below.
 In such situation,
 the family is parameterized by a "one-sided" range of $\theta$ of a
 fixed sign. In this note we consider generating measures with support
 bounded from above
 and our \CSK families  are parameterized by $\theta > 0$.
 Part of our motivation is to
 admit  the free additive $1/2$-stable law \cite[page
 1054]{Bercovici:1999} as a generating measure;  for exponential
 families the celebrated inverse gaussian law is $1/2$-stable and
 corresponds to cubic variance function.

Throughout the paper,
  \begin{equation}\label{B(nu)}
  B=B(\nu)=\max\{0,\sup \mbox{supp}(\nu)\}
\end{equation}
denotes the (sometimes strict) non-negative upper bound for the support of $\nu$. Occasionally we will want to explain how the formulas change for the "dual case" of measures with support bounded from below.
Then we will use $A=A(\nu)=\min\{0,\inf \mbox{supp}(\nu)\}$. Of course, $\mbox{supp}(\nu)\subset [A,B]$ when both expressions are finite.
 \begin{definition}
  Suppose $\nu$ is a   non-degenerate (i.e. not a point mass)
 probability measure with support  bounded from above.
For  $0\leq \theta B(\nu)<1$ let
\begin{equation*}  
M(\theta)=\int \frac{1}{1- \theta x} \nu(dx).
\end{equation*}
The (one-sided) Cauchy-Stieltjes kernel family generated by
$\nu$ is the family of probability measures
\begin{equation}\label{Kernel family}
\mathcal{K}_+(\nu)=\left\{P_\theta(dx)=\frac{1}{M(\theta)(1-\theta
x)}\nu(dx): 0< \theta<\theta_+\right\},
\end{equation}
where $\theta_+=1/B(\nu)$ (here $1/0$ is interpreted as $\infty$).
\end{definition}
To simplify the statements such as reparameterization  \eqref{Kernel family2}  of $ \mathcal{K}_+(\nu)$, we chose to exclude $\theta=0$ so that $\nu\not\in\mathcal{K}_+(\nu)$.
Alternatively, one could include $\nu$ as $P_0$ in the family $ \mathcal{K}_+(\nu)$; then one would need to write \eqref{Kernel family2}
 with  the left-closed domain of means $[m_0,m_+)$ and one would need to allow also the extended value
$m_0=-\infty$.

Similarly, one may define the one-sided \CSK family for a generating measure $\nu$ with support
bounded from below.
Then the one-sided \CSK family $\mathcal{K}_-(\nu)$ is defined
for $\theta_-<\theta< 0$, where $\theta_-$ is either $1/A(\nu)$ or $-\infty$.
Finally, if $\nu$ has compact support, then the natural
domain for the parameter $\theta$ of the two-sided \CSK family
$\mathcal{K}(\nu)=\mathcal{K}_-(\nu)\cup \mathcal{K}_+(\nu)\cup\{\nu\}$ is
$\theta_-<\theta<\theta_+$. For definiteness, we concentrate on the
case of measures bounded from above, and we explicitly allow generating
measures with unbounded support, and perhaps without
moments.

\section{Parameterizations by the mean}
Since $\theta> 0$,  the mean $m(\theta)=\int x
P_\theta(dx)$ exists for all measures  in \eqref{Kernel family}.  A calculation
gives
\begin{equation}\label{L2m}
m(\theta)=\frac{M(\theta)-1}{\theta M(\theta)}.
\end{equation}
Indeed,
\begin{equation*}
m(\theta)
= \int_{(-\infty,b]}\frac{1-(1-\theta x)}{\theta
M(\theta)(1-\theta x)}\nu(dx)
=\frac{1}{\theta}-\frac{1}{\theta
M(\theta)}.
\end{equation*}

To verify that $\theta\mapsto M(\theta)$ is differentiable on
$(0,\theta_+)$, we check that one can differentiate under the
integral sign. For this, we first observe that since $\mbox{supp}(
\nu)\subset (-\infty,B]$, for $\theta\in(0,\theta_+)$ the expression
$1-\theta x$ is positive for all $x$ from the support of $\nu$. So
\begin{multline*}
 \int \frac{|x|}{|1-\theta x|^2}\nu(dx)
\leq
\frac{1}{\theta}\int \frac{|\theta x-1|+1}{|1-\theta x|^2}\nu(dx)
\\=\frac{1}{\theta}\int \frac{(1-\theta x)+1}{(1-\theta
x)^2}\nu(dx)
\leq \frac{M(\theta)}{\theta}+\frac{M(\theta)}{\theta(1-\theta
B)}<\infty.
\end{multline*}
Now fix $0<\alpha<\beta<\theta_+$. Then, since for $x$ in the
support of $\nu$, the map $\theta\mapsto \frac{\partial}{\partial
\theta}\left(\frac{1}{1-\theta x}\right)=\frac{x}{(1-\theta x)^2}$
is increasing on $(0,\theta_+)$, we have

$$\frac{x}{(1-\alpha x)^2}\leq \frac{x}{(1-\theta x)^2}\leq\frac{x}{(1-\beta
x)^2},$$ for all $\theta \in [\alpha,\beta]$. \\Define for $x$ in the support
of $\nu$,
$$g(x)=\frac{|x|}{(1-\alpha x)^2}+\frac{|x|}{(1-\beta x)^2}.$$
Then we have that $g\geq0$, $g$ is $\nu-$integrable, because $\alpha$ and
$ \beta$ are in $(0,\theta_+)$, and $\left|\frac{\partial}{\partial
\theta}\left(\frac{1}{1-\theta x}\right)\right|=\frac{|x|}{(1-\theta
x)^2}\leq g(x)$, for all $\theta \in [\alpha,\beta]$.
 With this, we have that $\theta\mapsto M(\theta)$ is differentiable on
$(0,\theta_+)$, and we can differentiate under the integral sign.
The reasoning from \cite{Bryc-06-08} can now be repeated  to show that
$m(\theta)$ is increasing. Indeed, differentiating \eqref{L2m} we get
$$m'(\theta)=\frac{M(\theta)+\theta M'(\theta)-(M(\theta))^2}{\theta^2(M(\theta))^2}.$$
Since $\nu$ is non-degenerate,
\begin{multline*}
M(\theta)+\theta M'(\theta)-(M(\theta))^2\\
=\int\frac{1}{(1-\theta x)^2}\nu(dx)-\left(\int\frac{1}{1-\theta
x}\nu(dx)\right)^2>0
\end{multline*}
for all $0<\theta<\theta_+$. Thus the function $\theta\mapsto
m(\theta)$ is increasing on $(0,\theta_+)$. Denoting by $\psi$  the
inverse function, we are thus lead to parametrization  of
$\mathcal{K}_+(\nu)$  by the mean,
\begin{equation}\label{Kernel family2}
\mathcal{K}_+(\nu)=\left\{Q_m(dx)=P_{\psi(m)}(dx): m\in
(m_0,m_+)\right\},
\end{equation}
where the so called domain of means $(m_0,m_+)$ is the image under
 $\theta\mapsto m(\theta)$ of the interval $(0,\theta_+)$. It is
clear that $m_+\leq \sup\mbox{supp}(\nu)$ and
$$m_0=\lim_{\theta\searrow 0} m(\theta)=
\int x \nu(dx)\geq -\infty.$$
We will give alternative
representations for $m_0$ and $m_+$ later in the paper, see
 Proposition \ref{P-R}(iii) and Proposition \ref{P-DM}.

\section{The pseudo-variance function}
The variance function
\begin{equation}
  \label{Def Var}
  V(m)=\int (x-m)^2 Q_m(dx)
\end{equation} is the fundamental concept of the theory of
 exponential families, and also of the theory of \CSK families as
 presented in \cite{Bryc-06-08}.  Unfortunately, if  $\nu$
does not have the first moment (which is the case for free
 $1/2$-stable laws), all measures in the  \CSK family generated by
 $\nu$ have infinite variance.
We therefore introduce the following substitute which coincides with the variance function when the mean of $\nu$ is zero.
\begin{definition}
For $m\in(m_0,m_+)$,   the {\em pseudo-variance function} is defined
as
\begin{equation}\label{m2v}
\V(m)=m\left(\frac{1}{\psi(m)}-m\right).
\end{equation}
\end{definition}
Note that  the pseudo-variance function  may take negative values:
$\V(m)$ is negative  for all $m_0<m<0$, as in this case $1/\psi(m)-m>0$.

Since \eqref{m2v} is difficult to interpret, we will clarify in Proposition \ref{P.3.2}   how the pseudo-variance is related to the variance function when $\int |x|\nu(dx)<\infty$. To derive this relation, we need first to   point out how
the pseudo-variance function enters into the
representation of a \CSK family as a free-exponential or
$q$-exponential family, compare \cite{Bryc-Ismail-05}.
  \begin{proposition}
 Suppose $\V$  is a pseudo-variance of the \CSK family
 $\mathcal{K}(\nu)$ generated by a probability measure $\nu$ with
 support in $(-\infty,b]$ for some $b\in\RR$.  The explicit
 parametrization of  $\mathcal{K}_+(\nu)$ by the mean
\eqref{Kernel family2} is as follows:
   \begin{equation}
  \label{F(V)}
Q_m(dx)=\frac{\V(m)}{\V(m)+m(m-x)}\nu(dx).
\end{equation}

 \end{proposition}
 \begin{proof} For completeness, we include a version of the argument
 implicit in \cite{Bryc-06-08}: If $\theta=\psi(m)=m/(m^2+\V(m))$, then from
 \eqref{m2v} we get
  $$\frac{1}{M(\theta)(1-\theta x)}=\frac{m^2+\V(m)}{M(\theta)(\V(m)+m(m-x))}.$$
From \eqref{L2m}, $1/M(\theta)=1-\psi(m)m=\V(m)/(m^2+\V(m))$,
proving \eqref{F(V)}.
\end{proof}

The following result shows that the pseudo-variance functions are closely related to the variance functions.
\begin{proposition}\label{P.3.2}
 If $\nu$ is non-degenerate with support bounded from above, and
 with finite first moment  $m_0=\int x \nu(dx)$ then the variance function
 $V$ as defined by \eqref{Def Var} exists, and
 $$\V(m)=\frac{m}{m-m_0}V(m)$$
  for $m\in(m_0,m_+)$.   In particular, if $m_0=0$ then $\V(m)=V(m)$ on $(0,m_+)$.
\end{proposition}
\begin{proof} For completeness, we include a version of the argument
 presented in \cite{Bryc-06-08}. A  different argument based on
 \eqref{BIS} is presented in \cite{Bryc-Ismail-05}.

If $\nu$ has the first moment, then $m_0=\int x\nu(dx)$ is a real
number and $x(x-m)$ is integrable with respect to measures
$Q_m\in\mathcal{K}(\nu)$ for all $m\in(m_0,m_+)$. The following
calculation gives the answer:
\begin{multline*}
V(m)=\int x(x-m)Q_m(dx)=\int \frac{x(x-m)\V(m)}{\V(m)+m(m-x)}\nu(dx)\\
=
\frac{\V(m)}{m}\int \frac{xm (x-m)}{\V(m)+m(m-x)}\nu(dx)\\
=
\frac{\V(m)}{m}\int \frac{ x \V(m)-x(\V(m)+m(m-x) )}{\V(m)+m(m-x)}\nu(dx)\\
=\frac{\V(m)}{m}\left(\int x Q_m(dx)-\int x
\nu(dx)\right)=\frac{\V(m)}{m}(m-m_0).
\end{multline*}

\end{proof}

\begin{remark}  If
  $\nu$ has also finite second moment and $m_0\ne 0$,  then $\V(m)$
 has a simple pole at $m=m_0$. In Corollary \ref{C-V} we will verify
 that $m/\V(m)$ is positive, continuous, and increasing on
 $(m_0,m_+)$, so this must be the only pole.
\end{remark}

\begin{remark}
Using the finite difference operator $$(\Delta_m
f)(m):=\frac{f(m)-f(0)}{m},$$ it is straightforward to verify that
when $0\in(m_0,m_+)$ the density $g(m,x):=\frac{\V(m)}{\V(m)+m(m-x)} $ satisfies the
equation
\begin{equation}\label{BIS}
\Delta_m g(m,x)=\frac{x-m}{\V(m)}g(m,x).
\end{equation}
This is a finite-difference  analog of the differential equation
satisfied by the density of an exponential family \cite{Wedderburn74}. In
\cite{Bryc-Ismail-05}, this equation is used to verify directly from
 \eqref{F(V)} that $m$ is the mean and that $\V(m)$ is the variance
 function when $\int x\nu(dx)=0$.
 \end{remark}

In principle, the generating measure $\nu$ can be determined from the
 pseudo-variance function $\V$ by the following method.  Given
 $\V(m)$, solve \eqref{m2v} for
$\psi$, find the inverse function $m(\theta)$, and solve \eqref{L2m}
for $M(\theta)$. This effectively determines the distribution via
Stieltjes inversion formula \eqref{SIF}.
 To ensure that this
procedure indeed works we need several technical results.

\begin{proposition}\label{P.z}
 Suppose $\V$ is the pseudo-variance function of a \CSK family
 generated by a probability measure with  $B=B(\nu)<\infty$
 (recall \eqref{B(nu)}).  Let
\begin{equation}\label{z2m}
 z=z(m)=m+\frac{\V(m)}{ m}.
 \end{equation}
 Then $m\mapsto  z(m)$ is
 continuous, strictly decreasing on $(m_0,m_+)$,
 $z(m)>0$ on  $(m_0,m_+)$,  $z(m)\nearrow \infty$ as $m\searrow m_0$,
 and $z(m)\searrow B$ as $m\nearrow m_+$.
\end{proposition}
 \begin{proof}Rewriting \eqref{m2v} we see that $z(m)=1/\psi(m)$ is
 strictly decreasing,  positive and increases without bound as
 $m\searrow m_0$.
 Clearly,  as an inverse of the differentiable function
 $\theta\mapsto m(\theta)$, function $\psi(m)$ is continuous.
\end{proof}

\subsection{Cauchy transform}
The Cauchy transform of a probability measure $\nu$  on Borel sets
of $\RR$ is  an analytic mapping $G$ from the upper complex
half-plane $\CC^+$ into the lower half-plane $\CC^-$ given by
\begin{equation}\label{G-transform}
G_\nu(z)=\int\frac{1}{z-x}\nu(dx).
\end{equation}
(We will drop subscript $\nu$ when the measure is clear from context.) It is known (see \cite{Akhiezer:1965} or  \cite[Proposition
5.1]{Bercovici:1993}) that an analytic function $G:\CC^+\to\CC^-$ is
a Cauchy transform if and only if
\begin{equation}
\lim_{t\to\infty} it {G(it)}=1
\end{equation}
and  that the corresponding probability measure $\nu$ is determined
uniquely from the Stieltjes inversion formula
\begin{equation}\label{SIF}
\nu(dx)=\lim_{\eps\searrow 0} \frac{-\Im G(x+i\eps)}{\pi} dx,
\end{equation}
where the limit is in the sense of weak convergence of measures.
(See e.g., \cite[page 125]{Akhiezer:1965}.)

For a probability measure $\nu$  with support in $(-\infty,b]$,  $G_\nu$ is
analytic on the slit complex plane $\CC\setminus(-\infty,b]$. This
shows that  a probability measure  $\nu$ with support bounded from
above is determined uniquely by $G_\nu(z)$ on $z\in(b,\infty)$ for
some $b$.

Furthermore, $\lim_{z\to\infty }zG(z)=1$, and since $G$ is non-negative and decreasing on $(b,\infty)$, the limit
$\lim_{z\searrow b}G_\nu(z)$   exists as an extended number in $(0,\infty]$. It  will be convenient to write $1/G(b)\in[0,\infty)$ for the limit $\lim_{z\searrow b}1/G_\nu(z)$
even if  $G(z)$ is undefined at $z=b$.

The following shows how  the upper end of the domain of means is
related directly to Cauchy transform.
\begin{proposition}\label{P-DM}
If $\nu$ has support bounded from above with $B=B(\nu)<\infty$ given by \eqref{B(nu)}, then
$m_+=B-1/G(B)$. (Here $1/G(B):=\lim_{z\searrow B} 1/G(z)$ can be
zero.) \end{proposition}
\begin{proof} Since
$M(\theta)=\frac{1}{\theta} G\left(\frac{1}{\theta}\right)$
from \eqref{L2m} we get
$m_+=\lim_{\theta\nearrow 1/B}m(\theta)= \lim_{\theta\nearrow
1/B}\left(1/\theta -1/G(1/\theta)\right)=B-1/G(B)$.
\end{proof}

\begin{remark}[Precise domain of means]\label{R.PDM}
If $\nu$ is compactly supported and $A=\min\{0,\inf
 \mbox{supp}(\nu)\}$ and  $B=B(\nu)$, then the domain of means for
 the two sided \CSK family generated by $\nu$ is  the interval $(m_-,m_+)$ with
$m_-=A-1/G(A)$, $m_+=B-1/G(B)$. 
This gives a more precise information about the domain of means considered in
 \cite[Theorem 3.1]{Bryc-06-08}.
\end{remark}

From the fact that measures \eqref{F(V)} integrate to $1$, we get
the following (see \cite[Theorem 3.1]{Bryc-06-08}).
\begin{proposition}\label{P1.1}
Suppose $\V$  is the pseudo-variance function of the \CSK family
$\mathcal{K}_+(\nu)$ generated by a probability measure $\nu$ with
support bounded from above.
For $z$ given by \eqref{z2m}, the Cauchy-Stieltjes transform of
$\nu$ satisfies  \begin{equation}
  \label{G2V}
  G(z)=\frac{m}{\V(m)}.
\end{equation}
In particular,
$0<\frac{m}{\V(m)}<G(B(\nu))$. The generating measure
$\nu$ is determined uniquely by the pseudo-variance function
$\V(\cdot)$.
\end{proposition}
\begin{proof}  Integrating each measure in \eqref{F(V)}, we get
 \eqref{G2V}. Since $G$ is positive and decreasing on
 $(B(\nu),\infty)$, from  \eqref{G2V} we
see that $0<m/\V(m)<G(B(\nu))$ on $(m_0,m_+)$.
 By Proposition \ref{P.z}, this determines $G(z)$ on an interval.
 Thus by the uniqueness of analytic extension, this determines $G(z)$
 for all $z\in\CC^+$, and hence it determines $\nu$ via \eqref{SIF}.
\end{proof}
\begin{corollary}\label{C.V/m}
If $\V$ is a pseudo-variance function of a \CSK family generated by
probability measure with support bounded from above, then
$\frac{m}{\V(m)}\to 0$  and $\frac{m^2}{\V(m)}\to 0$ as $m\searrow
m_0$.
\end{corollary}
\begin{proof} 
By Proposition \ref{P.z}, $z(m)\to \infty$. So \eqref{G2V} and the properties of the Cauchy transform imply that
 $G(z(m))=m/\V(m)\to 0$ and $zG(z)=1+m^2/\V(m)\to 1$.
\end{proof}

\begin{remark} Proposition \ref{P1.1} shows that pseudo-variance
 function $m\mapsto \V(m)$
contains more information than the variance function $m\mapsto V(m)$: $\V$ is
 defined for $\nu$ without moments, and it determines the generating
 measure  $\nu$ without the need to supply its mean. In particular,
 for a two-sided family with the domain of
means $(m_-,m_+)$ as in Remark \ref{R.PDM}, its pseudo-variance
 function determines also $m_0=\int x\nu(dx)$.
\end{remark}

The following technical result is needed later on to verify that the
$R$-transform is strictly increasing on an interval. (The inequality
gives the lower bound for the difference quotient of $F=1/G$ on
$(b,\infty)$; compare \cite[Proposition 2.1]{Maassen:1992}.)
 \begin{proposition}\label{P-G}
 If $\nu$ is a non-degenerate probability measure with the support
 bounded above by $b\in\RR$, then for $b<z_1<z_2$,
 \begin{equation}\label{G-inequality}
 G(z_1)-G(z_2)>(z_2-z_1)G(z_1)G(z_2).
\end{equation}
\end{proposition}
\begin{proof}
We have
\begin{equation}\label{Delta_G}
G(z_1)-G(z_2)=(z_2-z_1)\int_{(-\infty,b]}\frac{1}{(z_1-x)(z_2-x)}\nu(dx).
\end{equation}
Note that for $x,y\leq b$,
\begin{multline*}
\left(\frac{1}{z_1-x}-\frac{1}{z_1-y}\right)\left(\frac{1}{z_2-x}-\frac{1}{z_2-y}\right)\\
= \frac{(y-x)^2}{(z_1-x)(z_1-y)(z_2-x)(z_2-y)}\geq 0.
\end{multline*}
Choose $a<b$ such that $\nu((a,b])>0$. Since $\nu$ is
non-degenerate,
$$\iint_{(a,b]\times(a,b]}(y-x)^2 \nu(dx)\nu(dy)>0.$$
Therefore,
\begin{multline*}
2 \int_{(-\infty,b]}\frac{1}{(z_1-x)(z_2-x)}\nu(dx) - 2 G(z_1)G(z_2)\\
= \iint_{(-\infty, b]\times (-\infty, b] }
 \left(\frac{1}{z_1-x}-\frac{1}{z_1-y}\right)\left(\frac{1}{z_2-x}-\frac{1}{z_2-y}\right)\nu(dx)\nu(dy)
\\ =
\iint_{(-\infty, b]\times (-\infty, b] }\frac{(y-x)^2}{(z_1-x)(z_1-y)(z_2-x)(z_2-y)}\nu(dx)\nu(dy)\\
\geq \frac{1}{(z_2-a)^4} \iint_{(a,b]\times(a,b]}(y-x)^2
\nu(dx)\nu(dy)>0.
\end{multline*}
Since $z_2>z_1$, this together with \eqref{Delta_G} ends the proof.
\end{proof}

\subsection{Free convolution}

It is known, see \cite{Bercovici:1993}, that if $G$ is a Cauchy
transform of a probability measure, then
there exist $b>0$   such that $G$ is univalent in the domain
\begin{equation}\label{Gamma+}
\Gamma_{b}^+=\left\{z\in\CC^+:  \Re z>b, \Im z<  \Re z \right\}
\end{equation}
Since $G(\bar z)=\overline{ G(z)}$,  $G$ is also univalent in
 $\Gamma_{b}^-=\overline{\Gamma_{b}^+}$.
For measures with support bounded from above by $b>0$, $G$ is also
one-to-one on $(b,\infty)$. So increasing $b$ if necessary, we can
extend the region of univalence to 
\begin{equation}\label{Gamma}
\Gamma_{b}=\left\{z\in\CC: \Re z>b, |\Im z|<  \Re z \right\}.
\end{equation}
Then the composition-inverse function
$K_\nu(z)=G^{\langle-1\rangle}_\nu(z)$  exists and is analytic for
$z$ in the domain $G(\Gamma_b)$. The $R$-transform is an analytic
function  in the same region, and is defined by
\begin{equation}\label{R def} \R_\nu(z)=K_\nu(z)-1/z\;. \end{equation}
(A warning is in place: some authors
use $R(z)=z\R(z)$ as the $R$-transform!)
As an analytic function, $\R_\nu$ is determined uniquely by its
values on the interval $\RR\cap G( \Gamma_b)=(0,G(b))$. In fact,  on the
real line $\R_\nu$ is defined on a potentially larger interval
$(0,G(B(\nu)))$.

Our interest in the
$R$-transform stems from its relation to free convolution:  a
free convolution $\mu\boxplus\nu$ of probability measures $\mu,\nu$
on Borel sets of
the real line is a uniquely defined  probability measure $\mu\boxplus\nu$
such that
$$
\R_{\mu\boxplus\nu}(z)=\R_\mu(z)+\R_\nu(z)
$$
for all $z$ in an appropriate domain (see \cite[Section
5]{Bercovici:1993} for details; the exact form of this domain is not
relevant for us, as we will be working only
with the intervals in $\RR$
and then appeal to the uniqueness of analytic extension.)

For $\alpha>0$ we denote by $\nu^{\boxplus \alpha}$ the free
convolution power of a probability measure $\nu$, which is defined
by
\begin{equation}\label{R-power}
\R_{\nu^{\boxplus \alpha}}(z)=\alpha \R_\nu(z).
\end{equation}
Convolution power of order  $\alpha\in[1,\infty)$ exists by
\cite[Section 2]{Belinschi:2005}. Convolution power of order
$\alpha>0$ exists for $\boxplus$-infinitely divisible laws. The
following result lists properties of $R$-transform that we need.

\begin{proposition}\label{P-R}
Suppose $\V$  is a pseudo-variance function of the \CSK family
$\mathcal{K}_+(\nu)$ generated by a probability measure $\nu$ with
$b=\sup\mbox{supp}(\nu)<\infty$.
Then \begin{enumerate}
\item $\R_\nu$ is strictly increasing on $(0,G(b))$;
\item  For  $m\in(m_0,m_+)$,
 \begin{equation}\label{R formula}
\R_\nu\left(\frac{m}{\V(m)}\right)=m;
\end{equation}
\item  $\lim_{z\searrow 0}\R_\nu(z)=m_0\geq-\infty$;
\item
$\lim_{z\searrow 0}z\R_\nu(z)=0$. \end{enumerate}

(Of course,  the only new contribution of (iv) is  the case $m_0=-\infty$.)

\end{proposition}
\begin{proof}
(i) Choose $0<x_1<x_2<G(b)$. Then $\R_\nu(x_j)=K_\nu(x_j)-1/x_j$ are well
defined. Let $u_j=K_\nu(x_j)=\R_\nu(x_j)+1/x_j$ so that $x_j=G_\nu(u_j)$.
Clearly, $u_1>u_2>b$. Then \eqref{G-inequality} says $1/G
(u_1)-1/G(u_2)>u_1-u_2$ so
 $1/x_1-1/x_2>K_\nu(x_1)-K_\nu(x_2)$, i.e. $\R_\nu(x_1)<\R_\nu(x_2)$.

(ii) This is  the same as  \eqref{G2V}.

(iii) Since the limit exists by part (i), this is a consequence of
\eqref{R formula}.

(iv) By \eqref{R formula} with $z=m/\V(m)$, we have
 $z\R_\nu(z)=m^2/\V(m)\to 0$ as $m\to m_0$ by Corollary \ref{C.V/m}. Since
 $\R_\nu$ is increasing on $(0,G(b))$, this ends the proof.
\end{proof}

\begin{corollary}\label{C-V}
$m\mapsto m/\V(m)$ is strictly increasing and smooth function on $(m_0,m_+)$.
\end{corollary}
\begin{proof} We rewrite \eqref{R formula} as $m/\V(m)=\R_\nu^{\langle
 -1\rangle}(m)$, and use the fact that $\R$ is smooth and strictly increasing.
\end{proof}
It is worth mentioning here that one can use the $R$-transform to
determine the domain of means and the pseudo-variance function of a
\CSK family, and even to define them. In fact
$(m_0,m_+)=\R_\nu((0,G(B(\nu))))$, and $\V$ is nothing but the
function which gives $\R_\nu(z)/z$ as a function of $\R_\nu(z)$. For
example, if $\nu$ is the inverse semicircle law with
$\R_\nu(z)=-p/\sqrt{z}$, we have that
$\R_\nu(z)/z=\frac{(\R_\nu(z))^{3}}{p^2}$ so that the
pseudo-variance function of the generated \CSK family is equal to
$\frac{m^{3}}{p^2}$. This is the analogue of the fact that, for the
classical exponential families, the variance function is the
function which gives the second derivative of the cumulant function
in terms of the first derivative.

 \begin{proposition}\label{P-V-alpha}
 Let $\V_\nu$ be the pseudo-variance function of the one sided \CSK
 family generated by a probability measure $\nu$ with support bounded
 from above and with the mean $-\infty\leq m_0< \infty$.   Then for $\alpha>0$
 such that  $\nu^{\boxplus \alpha}$ is
defined, the support of  $\nu^{\boxplus \alpha}$ is bounded from
 above and for $m>\alpha m_0$ close enough to $\alpha m_0$,
  \begin{equation}\label{V-power}
 \V_{\nu^{\boxplus \alpha}}(m)=\alpha \V_\nu(m/\alpha).
\end{equation}
\end{proposition}
\begin{proof}
We first show that  the support of
$\nu^{\boxplus \alpha}$  is bounded from above. For functions with
 support bounded from above, $\R_\nu$ is univalent in a domain that
contains some open interval $(0,\delta)$. Therefore $\R_{\nu^{\boxplus \alpha}}$ is
 univalent in the same domain. This shows that $G_{\nu^{\boxplus
 \alpha}}$ is analytic
on a domain that contains $(c,\infty)$, where $c=K_{\nu^{\boxplus
 \alpha}}(\delta)$.
So the support of $\nu^{\boxplus \alpha}$ is bounded from above by
 $c$, see \cite[Proposition 6.1]{Bercovici:1993}.

From Proposition \ref{P-R}(iii) we see that the domain of
means for  one sided  \CSK generated by $\nu^{\boxplus \alpha}$
starts at  $\lim_{z\searrow 0} \R_{\nu^{\boxplus \alpha}}(z)=\alpha
m_0$.
 So for $m>\alpha m_0$ close enough to $\alpha m_0$ so that
 $m/\alpha \in (m_0,m_+)$ and $m/{\V_{\nu^{\boxplus
\alpha}}(m)}\in(0, G(B(\nu)))$ (recall Corollary \ref{C.V/m}) we can
 apply \eqref{R formula} and \eqref{R-power} to
see that
$$
\R_\nu\left(\frac{m}{\V_{\nu^{\boxplus \alpha}}(m)}\right)=
\frac{1}{\alpha}\R_{\nu^{\boxplus
\alpha}}\left(\frac{m}{\V_{\nu^{\boxplus
\alpha}}(m)}\right)=m/\alpha=\R_\nu\left(\frac{m/\alpha}{
\V_\nu(m/\alpha)}\right).
$$
From Proposition \ref{P-R}(i) we know that $\R_\nu$ is one-to-one on $(0,G(B(\nu)))$, so
$$
\frac{m}{\V_{\nu^{\boxplus \alpha}}(m)}=\frac{m/\alpha}{
\V_\nu(m/\alpha)},
$$
and formula \eqref{V-power} follows.
\end{proof}
We remark that the restriction of \eqref{V-power} to $m$ "close enough" to $\alpha m_0$ cannot be easily avoided, as we do not have a general formula for the upper end of the domain of means for $\nu^{\alpha \boxplus}$. (For the freely
$r$-stable laws the upper end of the domain of means is $\alpha^ r m_+$, so we do not expect a simple general formula.)

\subsection{Affine transformations}\label{Sect:Affine}
 Here we collect the formulas that describe the effects of applying
 an affine transformation to the generating measure.

For $\delta\ne 0$ and $\gamma\in\RR$, let $\varphi(\nu)$ be the
image of $\nu$ under the affine map $x\mapsto \frac{x-\gamma}{\delta}$. In
other words,  if $X$ is a random variable with  law $\nu$ then $\varphi(\nu)$
is the law of $(X-\gamma)/\delta$, or $\varphi(\nu)=D_{1/\delta}(\nu\star \delta_{-\gamma})$,
where $D_r(\mu)$ denotes the dilation of measure $\mu$ by a number $r\ne0$, i.e.
$D_r(\mu)(U)=\mu(U/r)$.

It is well known that  $G_{\varphi(\nu)}(z)=\delta G_\nu(\delta z+\gamma)$
and $\R_{\varphi(\nu)}(z)=1/\delta \R_\nu(z/\delta)-\gamma/\delta$.
 The effects of the affine transformation on the corresponding \CSK
 family are as follows:
\begin{itemize}
\item Point $m_0$ is transformed to
$(m_0-\gamma)/\delta$. In particular, if
$\delta<0$ then $\varphi(\nu)$ has support bounded from below and
then it generates the left-sided  $\mathcal{K}_-(\varphi(\nu))$.
\item
For $m$ close enough to $(m_0-\gamma)/\delta$ the pseudo-variance function is
\begin{equation}\label{V_phi}
\V_{\varphi(\nu)}(m)=\frac{m}{\delta(m\delta+\gamma)}\V_\nu(\delta m+\gamma).
\end{equation}
 In particular, if the variance function exists, then
$V_{\varphi(\nu)}(m)=\frac{1}{\delta^2}V_\nu(\delta m+\gamma)$.
\end{itemize}
A special case worth noting is the reflection $\varphi(x)= -x$. If
 $\nu$ has support bounded from above and its right-sided \CSK family
$\mathcal{K}_+(\nu)$ has domain of means $(m_0, m_+)$ and pseudo-variance function $\V_\nu(m)$, then
$\varphi(\nu)$ generates the left-sided \CSK family $\mathcal{K}_-(\varphi(\nu))$ with the domain of means
$(-m_+,-m_0)$ and the pseudo-variance function $\V_{\varphi(\nu)}(m)=\V_\nu(-m)$.

 \subsection{Reproductive property}

\begin{proposition}\label{C.4.3} If  $\V$   is a pseudo-variance function
of a \CSK family generated by a
probability measure $\nu$ with  support bounded from above, then for
 $\la\geq 1$  measure
$$\nu_{\la}:=D_{1/\la}(\nu^{\boxplus \la})$$ has also support bounded
 from above and there is $\delta>0$ such that the pseudo-variance
 function of the one sided \CSK family
generated by $\nu_\la$
is $\V(m)/\la$ for $m\in(m_0,m_0+\delta)$.

If $\nu$ is free-infinitely divisible, then the above holds for every
 $\la>0$. Conversely, if for every $\la>0$, there is
 $\delta=\delta(\la)>0$ such that $\V(m)/\la$ is a pseudo-variance function
 of some \CSK family on $(m_0,m_0+\delta)$, then
$\nu$ is free-infinitely divisible.
\end{proposition}
\begin{proof} This is closely related to Proposition \ref{P-V-alpha}
 and is similar to \cite[Proposition 4.3]{Bryc-Ismail-05}, see also
 \cite{Bryc-06-08}; the details are omitted.
\end{proof}

\section{Quadratic and cubic pseudo-variance functions}
In this section we review the description of \CSK families with
 quadratic variance functions, adding the precise domain of means,
 then we analyze  certain cubic variance functions and point out the
 reciprocity relation between these two cases.
  \subsection{\CSK families with quadratic variance functions} The
 generating measures of \CSK families with quadratic variance
 functions  $V(m)=a -b  m + c   m^2$ with $a>0$, i.e. with
the pseudo-variance functions of the form
  \begin{equation}\label{PV_Meixner}
  \V(m)=\frac{m(a - b  m + c   m^2)}{m-m_0}
\end{equation}
were determined in \cite{Bryc-Ismail-05}. (Up to affine
 transformations, it is enough to consider $m_0=0$ and $a=1$.)  These
 are the so called
  free Meixner laws (\cite{Anshelevich01,Saitoh-Yoshida01}). Since  free Meixner laws are compactly supported, they  generate two-sided \CSK families.
   Remark \ref{R.PDM} can be used to determine the precise domain of
 means which was not previously available, except for an ad-hoc
 technique for the semi-circle law in  \cite[Example 4.1]{Bryc-06-08}.

    Unsurprisingly, the domain of means ends at the rightmost atom of
 $\nu$ when there is one; but may fall strictly inside the support of $\nu$
 when there are no atoms in $(0,\infty)$. When $m_0=0$, a calculation gives     \begin{equation}\label{DM-meixner}
   m_+=\begin{cases}
   \frac{b-\sqrt{b^2-4ac}}{2c} & \mbox{ if either  $c>0$,
 $b>2\sqrt{ac}$ or  $-1\leq c<0$ }; \\
    a/b & \mbox{ if $c=0$ and $b>\sqrt{a}$} ;\\
       \sqrt{a/(1+c)}&  \mbox{when $\nu$ has no atoms in $(0,\infty)$}.\\
\end{cases}
\end{equation}
(We recommend  \cite{Saitoh-Yoshida01} for the treatment of atoms.)

\subsection{A class of families with cubic pseudo-variance
function}

 Next, we
 describe the class of Cauchy-Stieltjes kernel
families with pseudo-variance functions of the form
\begin{equation}\label{cubic}
 \V(m)=m (a m^2+ b
m + c),
\end{equation}
with $a>0$. This class is important because it is related to the
quadratic class by a relation of reciprocity which will be introduced
 in the next section. \\
Suppose that \eqref{cubic} is the
pseudo-variance function generated by a distribution $\nu$.   Then \eqref{z2m} is a quadratic equation for $m$, so we can use \eqref{G2V} to express $G$ as a function of real $z$ large enough.
By uniqueness of the
analytic extension we get
$$G_{\nu}(z)=\frac{b+1+2 a z-\sqrt{(b+1)^2+4 a (z-c)}}{2 (c+b z+a z^2)}$$
for all $z$ in the upper half plane $\CC^+$.
The Stieltjes inversion formula \eqref{SIF} gives
   \begin{multline}\label{gm_cubic}
   \nu(dx)=\frac{\sqrt{ 4 a c- (b+1)^2- 4 a x}}{2\pi (c+b x+ a x^2)}1_{(-\infty,c-(b+1)^2/(4a))}(x)\, dx
   \\+p(a,b,c)\delta_{-(b+\sqrt{b^2-4ac} )/(2a)},
\end{multline}
where the weight of the atom
$p(a,b,c)= 1-1/\sqrt{b^2-4ac}$ if  $b^2>4ac+1$, and is $0$ otherwise.
 (In particular for $c=0$, there is an atom at $-b/a$ if $b>1$, or an
 atom at $0$ if $b<-1$.)

From Proposition \ref{P-DM}, we see that the domain of means is
 $(-\infty, m_+)$ with
 $m_+=B(\nu)-1/G(B(\nu))$.   A calculation that goes over the cases
 when the support contains positive numbers shows that
 \begin{equation}\label{DM-cubic}
 m_+=\begin{cases}
 -\frac{1+b}{2a} &\mbox{ if   $c>0$, $-\sqrt{1+4ac}\leq b\leq 2\sqrt{ac}-1$};\\
 \\-\frac{(b+\sqrt{b^2-4ac} )}{2a} & \mbox{ if  $c> 0$ and  $b\leq
 -\sqrt{1+4ac}$};\\
 -\frac{b+1+\sqrt{(b+1)^2-4 a c}}{2 a} & \mbox{ if either $c\leq 0$,
 or $b>2\sqrt{ac}-1$ }.
 \end{cases}
\end{equation}

  The most interesting example in this class is the inverse
semicircle law with the pseudo-variance function $\V(m)=m^3/p^2$ which
corresponds to the case $a=1/p^2$, $b=0$ and $c=0$. We have that
$$
G_\nu(z)=\frac{p^2+2 z-p \sqrt{4 z+p^2}}{2z^2},
$$
and the density of $\nu$ is
\begin{equation}\label{half-stable}
f(x)=\frac{p\sqrt{-p^2-4x}}{2\pi
x^2}
\end{equation}
on $(-\infty,-p^2/4)$. This is a free 1/2-stable density, see
\cite[page 1054]{Bercovici:1999}, see also \cite{Perez-Abreu:2008}.

The domain of means is $(m_0,m_+)=(-\infty, -p^2)$; this can be read
 out either from
 $\psi(m)=\frac{p^2}{m(m+p^2)}$, or  from Proposition \ref{P-DM}
 where the last case of
 \eqref{DM-cubic} is relevant here.\\
Similarly, one can use (4.3), (4.4) and (4.5) to get the free
analogous of the five other members of the Letac-Mora class with
variance function of degree 3. Keeping the names given in [11], we
have
\begin{enumerate}
\item Free Abel (or Free Borel-Tanner)
$$ \nu(dx)=\frac{1}{\pi( 1-x ) \sqrt{-x}}\,1_{(-\infty,0)}(x)\, dx,$$
with pseudo-variance function $\V(m)=m^{2}(m-1)$ and  domain of the
means $(-\infty,0)$.\\

\item  Free Ressel (or Free Kendall)
$$ \nu(dx)=\frac{-1}{\pi x \sqrt{-1-  x}}\,1_{(-\infty,-1)}(x)\, dx,$$
with pseudo-variance function $\V(m)=m^{2}(m+1)$ and domain of the means $(-\infty,-2)$.\\

\item Free strict arcsine
 $$ \nu(dx)=\frac{\sqrt{ 3 -4x}}{2\pi (1+x^{2})}\,1_{(-\infty,3/4)}(x)\, dx,$$
 with pseudo-variance function $\V(m)=m(1+m^{2})$ and  domain of the means $(-\infty,-1/2)$.\\

\item Free large arcsine
 $$ \nu(dx)=\frac{r\sqrt{4-5r^{2} -4(1+r^{2})x}}{2\pi (x^2+r^{2}(1+x)^{2})}\,1_{S}(x)\, dx,$$
where $r>0$ and  $S=\left(-\infty,\frac{4-5r^{2}}{4(1+r^{2})}\right)$. The pseudo-variance function is
$$\V(m)=m\left(1+2m+\frac{1+r^{2}}{r^{2}}m^{2}\right),$$
 and domain of the means is $\left(-\infty,-\frac{3r^{2}}{2(1+r^{2})}\right)$, if
 $r^{2}\leq4/5$, and $\left(-\infty,-\frac{3r^{2}+r\sqrt{5r^{2}-4}}{2(1+r^{2})}\right)$, if
 $r^{2}>4/5$.\\

\item Free Tak\'{a}cs 
 $$ \nu(dx)=\frac{\sqrt{-5 r^2-2 r-1 -4r(1+r)x}}{2\pi r( 1+x)(1+(1+1/r) x)}\,1_{S}(x)\, dx+ (1-r)^+\delta_{-1},$$
where $r>0$ and $S=\left(-\infty,1-\frac{(1+3r)^{2}}{4r(1+r)}\right)$. The pseudo-variance function is 
$$\V(m)=m(1+m)\left(1+\frac{1+r}{r}m\right),$$
 and domain of the means is  $\left(-\infty,-\frac{1+3r+\sqrt{5r^{2}+2r+1}}{2(1+r)}\right)$.
\end{enumerate}
\subsection{Reciprocity}
The notion of reciprocity between two natural exponential families
is defined by a symmetric relation between the cumulant functions of
two generating measures (see
\cite[Section 5]{Let:Mor}). Similarly, we can define the
reciprocity between two Cauchy Stieltjes Kernel Families by a
relation between the $R$-transforms of the generating distributions.
\begin{definition} Suppose $\widetilde\nu, \nu$ are probability
 measures   with support bounded from above.
We say that the corresponding one-sided Cauchy-Stieltjes kernel
families $\mathcal{K}_+(\widetilde\nu)$ and $\mathcal{K}_+(\nu)$ are
reciprocal if $m_0:=\int x\nu(dx)$ and $\widetilde m_0:=\int
x\widetilde \nu(dx) $ are of opposite signs and there is $\delta>0$ such that
\begin{equation}\label{(5.1)}
\R_{\widetilde\nu}(z \left|\R_{\nu}(z)\right|)=-\frac{1}{\R_{\nu}(z)}
\end{equation}
for all $z$ in  $(0,\delta)$.
\\ In this case, we also say that the distributions
$\widetilde\nu$ and $\nu$
are reciprocal.
\end{definition}
We note that $\R_{\nu} $ is defined for $z>0$ small enough so by
Proposition \ref{P-R}(iv) both sides of the expression \eqref{(5.1)} are well
defined for  all $z\in(0,\delta)$ when $\delta>0$ is small enough.
We also remark that with $m_0:=\int x\nu(dx)\in[-\infty,\infty)$, in
\eqref{(5.1)}  we actually have
$$\left|\R_\nu(z)\right|=\begin{cases}
\R_\nu(z) & \mbox{ if $m_0\geq 0$} ;\\
-\R_\nu(z) & \mbox{ if $m_0<0$}
\end{cases}
$$
for  $z>0$ close enough to $0$.

Note that \eqref{(5.1)} is equivalent to
\begin{equation}\label{(5.2)}
\R_{\nu}(z'
\left|\R_{\widetilde\nu}(z')\right|)=-\frac{1}{\R_{\widetilde\nu}(z')}
\end{equation}
for all $z'>0$ small enough, so reciprocity is a symmetric relation.
Indeed, we first note that
\eqref{(5.1)} implies $m_0=-1/\widetilde m_0$ even if $m_0=0$ or $\widetilde m_0=0$. We consider
separately the cases $m_0\geq 0$ and $m_0<0$.

If $m_0\geq 0$, we set $z'=z \R_\nu(z)$. From Proposition \ref{P-R}(iv)  we see that $\R_\nu(z')$ is well defined for small enough $z>0$.  Then $\eqref{(5.1)}$ is
equivalent to
\begin{equation}
\R_{\widetilde\nu}(z')=-\frac{1}{\R_{\nu}(z)}\   \textrm{and} \ \
z=-z'\R_{\widetilde\nu}(z').
\end{equation}
Hence
\begin{equation}
\R_{\nu}(-z'\R_{\widetilde\nu}(z'))=-\frac{1}{\R_{\widetilde\nu}(z')}.
\end{equation}
Since $\R_{\widetilde\nu}(z')<0$, because $\widetilde m_0<0$, this
is nothing but $\eqref{(5.2)}$.

If $m_0<0$, we use the same reasoning using $z'=-z\R_\nu(z)$.

The reciprocity between $\mathcal{K}_+(\widetilde\nu)$ and
 $\mathcal{K}_+(\nu)$
may also be expressed using the variance functions. More precisely,
we have:
\begin{theorem}
 Let $\V_{\widetilde\nu}$ and $\V_\nu$ be the pseudo-variance
 functions of the right-sided
Cauchy-Stieltjes kernel families generated by $\widetilde\nu$ and $\nu$,
with  means $\widetilde m_0$ and $m_0$, respectively. Then
$\mathcal{K}_+(\widetilde\nu)$ and $\mathcal{K}_+(\nu)$ are
 reciprocal if and only
if
$m_0= -1/\widetilde m_0$ (it is understood that $-1/0=-\infty$), and
 \begin{equation}\label{(5.3)}
 \V_{\widetilde\nu}(m)=-|m|^{3}\V_\nu(-\frac{1}{m}).
\end{equation}
for all $m>\widetilde m_0$ close enough to $\widetilde m_0$.
\end{theorem}
\begin{proof}
Suppose $m_0=-1/\widetilde m_0$ and \eqref{(5.3)} holds for all
 $\widetilde m_0<m<M$. Decreasing $M$ if necessary, we may ensure
 that $1/m\in (m_0,m_+)$. Choose $z'>0$ such that
 $z'<M/\V_{\widetilde \nu}(M)$.Since $m\mapsto m/\V_{\widetilde
\nu}(m)$ is a continuous function, we can find $m'>\widetilde m_0$
 such that $z'=m'/\V_{\widetilde \nu}(m')$.  Let $m=-1/m'$ and
 $z:=m/\V_\nu(m)$. For our choice of $m,m'$ from \eqref{R formula} we
 get\begin{equation}\label{(5.4)}
\R_{\widetilde\nu}(z')=-\frac{1}{\R_{\nu}(z)}, \end{equation}
and to deduce \eqref{(5.1)}
we only need to note that
$
z'=z |\R_\nu(z)|
$. The latter is a consequence of \eqref{(5.3)} and  \eqref{R formula}.

To prove the converse implication, suppose  that \eqref{(5.1)} holds. Then taking the limit as
 $z\searrow 0$ we deduce that $m_0=-1/\widetilde m_0$.  Therefore for
 all $m> m_0$, close enough to $ m_0$
so that
$z:=m/\V_{\nu}(m)$ and $z':=z |\R_\nu(z)|$  are within the domain of
 \eqref{(5.1)}, we deduce that  \eqref{(5.4)} holds. As previously,
 for $m$ close enough to $m_0$, $z'$ is close enough to $0$ so that
 we can find $m'>\widetilde m_0$ such
that
$z'=m'/\V_{\widetilde\nu}(m')$. Then \eqref{(5.4)}  says that
 $mm'=-1$ (here we use \eqref{R formula}  again), so the identities
 $z=m/\V_{\nu}(m)$ and $z'=m'/\V_{\widetilde\nu}(m')$ imply \eqref{(5.3)}.

\end{proof}
\begin{remark} In particular, if $m_0\geq 0$ then \eqref{(5.3)} says
 that for 
 $m<0$ close enough to $m_0$
 we have
$\V_{\widetilde \nu}(m)=m^3 \V_\nu(-1/m)$.

Of course one can combine reciprocity with affine action
 $\varphi(x)=-x$.  Correspondingly, one can extend the definition of
 reciprocity to pairs  $\mathcal{K}_\pm(\nu)$ and
 $\mathcal{K}_\pm(\widetilde \nu)$.
\end{remark}

\subsubsection{Example}
As mentioned above, we have been interested in the class of the
Cauchy-Stieltjes kernel families with pseudo-variance functions of
the form \eqref{cubic}, because it is the class the Cauchy-Stieltjes
kernel families with pseudo-variance functions of degree three which
are obtained by reciprocity from the families with quadratic
variance functions. In fact the \CSK family generated with
 pseudo-variance \eqref{cubic}
 is reciprocal with the right-sided part of the
quadratic \CSK family with  (pseudo)-variance \eqref{PV_Meixner} for $m_0=0$.
 In particular, the semicircle family with variance
function equal to $\frac{1}{p^2}$, $m_+=1/p$  and the inverse
 semicircle family
with pseudo-variance function equal to $\frac{m^{3}}{p^2}$, $m_+=-p^2$ are
reciprocal. For $z>0$, their $R$-transforms
$\R_{\widetilde\nu}(z)=z/p^2$ and $\R_\nu(z)=-p/\sqrt{z}$ are related
 by 
formula \eqref{(5.1)}.

Comparing   \eqref{DM-meixner} and  \eqref{DM-cubic}   we see that
 for reciprocal families the upper ends of the domain of means  do
 not satisfy a simple relation.

We remark that for $c<0$ the reciprocal of the free-infinitely
 divisible cubic family is free-binomial law which is not
 free-infinitely divisible.

\subsection*{\bf Acknowledgements}
The authors thank M. Bo\.zejko for information about free
$1/2$-stable distribution and J. Weso\l owski for a copy of
 \cite{Wesolowski90}. We also thank two anonymous referees for their
  detailed reports that helped to improve the paper.




\end{document}